\documentclass{amsart}
\usepackage{amsmath,amsfonts,amsthm}
\usepackage{amssymb,latexsym}
\usepackage[colorlinks]{hyperref}
\usepackage{graphicx}
\usepackage[all]{xy}
\usepackage{layout}

\theoremstyle{plain}
\newtheorem{theorem}{Theorem}[section]
\newtheorem{lemma}[theorem]{Lemma}
\newtheorem{proposition}[theorem]{Proposition}
\newtheorem{corollary}[theorem]{Corollary}

\theoremstyle{definition}

\theoremstyle{remark}

\newtheorem*{remark}{Remark}
\numberwithin{equation}{section}

 \newcommand{\M}{\mathbb M_{k(n)}(E)}
\begin{document}

\title{Extensions and Dilations of module maps}

\author{Massoud Amini}

\address{Department of Mathematics\\ Faculty of Mathematical Sciences\\ Tarbiat Modares University\\ Tehran 14115-134, Iran}
\email{mamini@modares.ac.ir}

\address{School of Mathematics\\
Institute for Research in Fundamental Sciences (IPM)\\
Tehran 19395-5746, Iran}
\email{mamini@ipm.ir}

\thanks{The  author was partly supported by a grant from IPM (No. 94430215).}



\maketitle

\begin{abstract}
We study completely positive
module maps on $C^{*}$-algebras which are $C^*$-module over another $C^*$-algebra
with compatible actions. We extend several well known dilation and extension results to this setup, including the Stinespring dilation theorem and Wittstock, Arveson, and Voiculescu extension theorems.
\end{abstract}

\section{Introduction}\label{s1}

For the rest of this paper, we fix a $C^*$-algebra $\mathfrak A$ and let $A$ be a $C^*$-algebra and a left Banach
$\mathfrak A$-module (that is, a bi-module with contractive left action) with compatible conditions,
\begin{equation*}
(\alpha\cdot a)^*=\alpha^*\cdot a^*,\,\, \alpha\cdot (ab)=(\alpha\cdot a)b,\,\, (\alpha\beta)\cdot  a=\alpha\cdot(\beta\cdot a),
\end{equation*}
for each $a,b\in A $ and $\alpha, \beta\in \mathfrak A.$ In this case, we  say that $A$ is a  $\mathfrak A$-$C^*$-module, or simply a $C^*$-module (it is then understood that the algebra and module structures on $A$ are compatible in the above sense). A $C^*$-subalgebra which is also an $\mathfrak A$-submodule is simply called a $C^*$-submodule.

Let $Z(A)$ be the center of $A$. An element $a\in A$ is called $\mathfrak A$-central if $\alpha\cdot a\in Z(A)$, for each $\alpha\in \mathfrak A$. We say that $\mathfrak A$ acts centrally on $A$, or $A$ is a central $\mathfrak A$-module, if $Z(A)$ is a submodule of $A$. When $\mathfrak A$ is a unital $C^*$-algebra, we say that $A$ is unital (as an  $\mathfrak A$-module), if $A$ is a unital $C^*$-algebra and a neo-unital $\mathfrak A$-module. In this case, $A$ is a central $\mathfrak A$-module if and only if $1_A$ is a central element of $A$. this is also equivalent to the following compatibility condition:
$$(\alpha\cdot a)(\beta\cdot b)=(\alpha\beta)\cdot (ab),$$
for each $\alpha, \beta\in \mathfrak A$, each $a\in Z(A)$, and $b\in A$.

In some cases we have to work with operator $\mathfrak A$-modules with no algebra structure (and in particular with certain Hilbert $\mathfrak A$-modules). If $E, F$ are operator $\mathfrak A$-modules, a module map $\phi: E\to F$ is a continuous linear map which preserves the left $\mathfrak A$-module action.

Throughout this paper, we use the notation $\mathbb B(X)$ to denote the set of bounded (adjointable) linear operators on a Hilbert space (Hilbert $C^*$-module). Also the center of a $C^*$-algebra $D$ is denoted by $\mathcal Z(D)$.

\section{unitzation}

We freely use the abbreviations and notations of \cite{bo}, in particular, c.p., u.c.p., and c.c.p. stand for completely positive, unital completely positive, and contractive completely positive, respectively.

Let $A$ be an  $\mathfrak A$-module. We write $\mathfrak A\cdot A$ for the closed linear span of the set of elements of the form $\alpha\cdot a$, with $\alpha\in\mathfrak A, \ a\in A$. We say that $A$ is neo-unital if $\mathfrak A\cdot A=A$. In this case, if $\mathfrak A$ is a unital $C^*$-algebra, then $1_{\mathfrak A}\cdot a=a$, for $a\in A$.

When $\mathfrak A$ is a unital $C^*$-algebra, then $\tilde A:=A\times \mathfrak A$ is called the (minimal) unitization of $A$. It is a unital $C^*$-algebra under the norm $\|(a,\alpha)\|=\sup\{\|ab+\alpha\cdot b\|: b\in A, \|b\|=1\}$ and multiplication
$$(a,\alpha)(b,\beta)=(ab+\alpha\cdot b+\beta\cdot a, \alpha\beta),$$
with unit $(0, 1_{\mathfrak A})$. Then $A$ is identified with a closed ideal in $\tilde A$. The module action of $\mathfrak A$ on $\tilde A$ is defined by $\alpha\cdot (a,\beta)=(\alpha\cdot a, \alpha\beta)$. Clearly $\tilde A$ is unital as a $\mathfrak A$-module. One should note that $\tilde A$ is not the same as the direct sum $A\oplus \mathfrak A$ (whose elements are denoted by $a\oplus\alpha$).

\begin{lemma} \label{exp}
Let  $\theta: A\to B$ be a c.c.p. module map. If there is a sequence  of operator subsystems and submodules $E_n\subseteq \M$ with a sequence of c.c.p. projections $P_n: \M\to E_n$, and there are c.c.p. module maps $\varphi_n: A\to E_n$ and $\psi_n: E_n\to B$ such that $\psi_n\circ\varphi_n\to \theta$ in point-norm topology, then $\theta: A\to B$ is $E$-nuclear.
\end{lemma}
\begin{proof}
Let $\iota_n: E_n\hookrightarrow \M$ be the c.c.p. embedding and note that $P_n\circ\iota_n=id_{E_n}$, hence $(\psi_n\circ P_n)\circ(\iota_n\circ\varphi_n)\to \theta$ in point-norm topology, as required.
\end{proof}

The next lemma extends \cite[2.2.1-2.2.4]{bo} with similar proofs (except that here we should also take care of the module actions).

\begin{lemma}\label{unital}
Let $\mathfrak A$ be a unital $C^*$-algebra and $\theta: A\to B$ be a c.c.p. module map.

$(i)$ If $B$ is unital and central $\mathfrak A$-module, $\theta: A\to B$ extends to a u.c.p. module map
$$ \tilde\theta: \tilde A\to B;\ \ \tilde\theta(a, \alpha):=\theta(a)+ \alpha\cdot 1_B.$$
In general, $\theta: A\to B$ extends to a u.c.p. module map
$$ \tilde\theta: \tilde A\to \tilde B;\ \ \tilde\theta(a, \alpha):=(\theta(a), \alpha\cdot 1_{\tilde B}).$$

$(ii)$ If $A$ and $B$ are unital, $B$ is central and $1_A$ is in the multiplicative domain of $\theta$, then $\theta: A\to B$ extends to a u.c.p. module map
$$ \tilde\theta: A\oplus \mathfrak A\to B;\ \ \tilde\theta(a\oplus\alpha):=\theta(a)+ \alpha\cdot (1_B-\theta(1_A)).$$

$(iii)$ If $\theta$ is $E$-nuclear, then  $\tilde\theta$ is $(E\oplus \mathfrak A)$-nuclear in all the above cases.
\end{lemma}
\begin{proof}
$(i)$ We have to show that $\tilde \theta$ is c.p. map. The action of $\mathfrak A$ on $A$ extends to an action on $A^{**}$ which satisfies the compatibility conditions, and $A^{**}$ is a neo-unital $\mathfrak A$-module. Thus $\tilde A$ could be identified with $A+\mathfrak A\cdot 1_{A^{**}}\subseteq A^{**}$, and $\tilde\theta$ is the restriction of the c.p. map $\theta^{**}$, and so  $\tilde\theta: \tilde A\to B^{**}$ is a c.p. map. Next we have the matrix equation
$$\tilde\theta_n[a_{ij}+\alpha_{ij}\cdot 1_{A**}]-(\theta^{**})_n[a_{ij}+\alpha_{ij}\cdot 1_{A**}]=[\alpha_{ij}\cdot 1_{B}]diag\big(1_B-\theta^{**}(1_{A^{**}})\big).$$
If $[a_{ij}+\alpha_{ij}\cdot 1_{A**}]\in \mathbb M_n(\tilde A)_{+}$, then $[\alpha_{ij}\cdot 1_{A^{**}}]\in \mathbb M_n(A^{**})_{+}$, since $\mathbb M_n(\tilde A\cdot 1_{A^{**}})\subseteq \mathbb M_n(A^{**})$ could be identified with the quotient of $\mathbb M_n(\tilde A)$ by the closed ideal $\mathbb M_n(\tilde A)$. Thus
$$[\alpha_{ij}\cdot 1_{B}]=(\tilde\theta)_n[\alpha_{ij}\cdot 1_{A**}]\in \mathbb M_n(B^{**})_{+}.$$
On the other hand, by the compatibility conditions,
\begin{align*} (\alpha_{ij}\cdot 1_{B})\big(1_B-\theta^{**}(1_{A^{**}})\big)&=(\alpha_{ij}\cdot 1_{B})\big(1_{\mathfrak A}\cdot(1_B-\theta^{**}(1_{A^{**}}))\big)\\
 &=(\alpha_{ij}1_{\mathfrak A})\cdot \big(1_{B}(1_B-\theta^{**}(1_{A^{**}}))\big)\\
 &=(1_{\mathfrak A}\alpha_{ij})\cdot \big((1_B-\theta^{**}(1_{A^{**}}))1_{B}\big)\\
  &=\big(1_B-\theta^{**}(1_{A^{**}})\big)(\alpha_{ij}\cdot 1_{B}),
  \end{align*}
 hence the positive matrices $[\alpha_{ij}\cdot 1_{B}]$ and $diag\big(1_B-\theta^{**}(1_{A^{**}})\big)$ commute and their product is positive in $\mathbb M_n(B^{**})$. Therefore,
$$\tilde\theta_n[a_{ij}+\alpha_{ij}\cdot 1_{A**}]\geq(\theta^{**})_n[a_{ij}+\alpha_{ij}\cdot 1_{A**}]\geq 0,$$
as required.

$(ii)$ Without loss of generality we may assume that $\theta$ has dense range. Since $1_A$ is in the multiplicative domain of $\theta$, $\theta(1_A)\in Z(B)$. We have $1_B\geq\|\theta\|1_A\geq \theta(1_A)$. Choose $x\in Z(B)$ with $ 1_B- \theta(1_A) =xx^*$ and let $\alpha=\beta\beta^*\geq 0$ in $\mathfrak A$, then
$$\alpha\cdot \big(1_B- \theta(1_A)\big)=(\beta\beta^*)\cdot(xx^*)=(\beta\cdot x)(\beta\cdot x)^*\geq 0,$$
in $B$. This means that the map $\alpha\mapsto \alpha\cdot \big(1_B- \theta(1_A)\big)$ is positive from $\mathfrak A$ to $A$. A similar argument, this time applied to the positive matrices in $\mathbb M_n(\mathfrak A)$ shows that the map is indeed c.p. Now the sum of c.p. maps is again c.p.

$(iii)$ Let  $\varphi_n: A\to \M$ and $\psi_n: \M\to B$ be c.c.p. maps with $\psi_n\circ\varphi_n\to \theta$ in point-norm topology. In $(i)$, extend $\varphi_n$ to u.c.p. map $\tilde\varphi_n: \tilde A\to \M\oplus \mathfrak A$ and $\psi_n$ to u.c.p. map $\tilde\psi_n:  \M\oplus \mathfrak A\to B$. Now the canonical projection
$$P_n: \mathbb M_{k(n)}(E\oplus \mathfrak A)\to \M\oplus \mathfrak A$$
is a c.c.p. map and Lemma \ref{exp} applies. In $(ii)$, we extend $\varphi_n$ to u.c.p. map $\tilde\varphi_n: \tilde A\to \M\oplus \mathfrak A$ and $\psi_n$ to u.c.p. map $\tilde\psi_n:  \M\oplus \mathfrak A\to \tilde B$ and argue as above.
\end{proof}

In part $(ii)$, the condition that $1_A$ is in the multiplicative domain of $\theta$ could be replaced by any of the conditions that $\theta(1_A)$ is a projection or a unitary (both of these imply that $1_A$ is in the multiplicative domain of $\theta$, c.f. \cite{ch}) or $\theta(1_A)$ is in the center of $B$. The next lemma extends \cite[2.2.5]{bo}.

\begin{lemma} \label{in}
If $\mathfrak A$ is a unital $C^*$-algebra, $E$ is a unital injective $C^*$-algebra and a unital and central left $\mathfrak A$-module with compatible actions, $A$ is unital and $\tilde\varphi: A\to \mathbb M_n(E)$ is a c.p. module map, then there is a u.c.p. module map $\varphi: A\to \mathbb M_n(E)$ such that
$$\tilde\varphi(a)=\tilde\varphi(1_A)^{\frac{1}{2}}\varphi(a)\tilde\varphi(1_A)^{\frac{1}{2}},$$
for $a\in A$.
\end{lemma}
\begin{proof}
Consider a faithful embedding $E\subseteq \mathbb B(H)$ such that $E$ contains the identity of $\mathbb B(H)$. If $u:=\tilde\varphi(1_A)\in\mathbb M_n(E)\subseteq \mathbb M_n(\mathbb B(H))=\mathbb B(H^n)$ is invertible in $\mathbb B(H^n)$, we  put $\varphi_1(a)=u^{-\frac{1}{2}}\tilde\varphi(a)u^{-\frac{1}{2}}$. For each $a\in A$ and $\alpha\in\mathfrak A$,
$$u(\alpha\cdot I_n)=(1_{\mathfrak A}\cdot u)(\alpha\cdot I_n)=(1_{\mathfrak A}\alpha)\cdot (uI_n)=(\alpha1_{\mathfrak A})\cdot (I_nu)=(\alpha\cdot I_n)u,$$
where $I_n$ is the identity of $\mathbb M_n(E)$. Hence
$$\varphi_1(\alpha\cdot a)=u^{-\frac{1}{2}}\tilde\varphi(\alpha\cdot a)u^{-\frac{1}{2}}=u^{-\frac{1}{2}}(\alpha\cdot I_n) \tilde\varphi(a)u^{-\frac{1}{2}}=(\alpha\cdot I_n)u^{-\frac{1}{2}} \tilde\varphi(a)u^{-\frac{1}{2}}=\alpha\cdot\varphi_1(a),$$
and $\varphi_1: A\to \mathbb M_n(E)$ is a u.c.p. module map.

In general, let $p: H^n\to \ker\tilde\varphi(1_A)$ and put $p^\perp=1-p$, then for $0\leq a\leq 1_A$ and $\zeta\in\ker\tilde\varphi(1_A)$ we have
$$-\|\varphi(a)^{\frac{1}{2}}\zeta\|^2=\langle\big(\tilde\varphi(1_A)-\tilde\varphi(a)\big)\zeta,\zeta\rangle\geq 0,$$
thus $\zeta\in\ker\tilde\varphi(a)$, in particular, $p$ acts as identity on the image of $\tilde\varphi(a)$, that is, $\tilde\varphi(a)=p^\perp\tilde\varphi(a)=\tilde\varphi(a)p^\perp$, which then holds for any $a\in A$. We let $\mathfrak A$ act on $H$ by $\alpha\cdot \xi=(\alpha\cdot 1_E)(\xi),$ for $\alpha\in\mathfrak A$ and $\xi\in H$. This gives a canonical module structure on $\mathbb B(H)$ such that $\tilde\varphi: A\to p^\perp \mathbb B(H^n)p^\perp$ is a c.p. module map. By the first part of the proof, there is a u.c.p. module map $\varphi_1: A\to p^\perp\mathbb B(H^n)p^\perp$ satisfying $u^{\frac{1}{2}}\varphi_1(a)u^{\frac{1}{2}}=p^\perp\tilde\varphi(a)p^\perp$, for $a\in A$. Put $\varphi_2(a):=\varphi_1(a)\oplus\omega(a)p$, for some c.p. module map $\omega: A\to \mathfrak A$, then $\varphi_2: A\to \mathbb B(H^n)$ is a u.c.p. module map.

By the injectivity of $E$, there is a conditional expectation $\mathbb E: \mathbb B(H)\to E$. Let $\mathbb E_n: \mathbb B(H^n)\to \mathbb M_n(E)$ be the amplification of $\mathbb E$ and put $\varphi=\mathbb E_n\circ\varphi_2$. For $a\in A$ and $\alpha\in\mathfrak A$,
$$\varphi(\alpha\cdot a)=\mathbb E_n\big(\varphi_2(\alpha\cdot a)\big)=\mathbb E_n\big((\alpha\cdot I_n)\varphi_2(a)\big)=(\alpha\cdot I_n)\mathbb E_n\big(\varphi_2(a)\big)=\alpha\cdot\varphi(a),$$
and $\varphi: A\to \mathbb M_n(E)$ is a u.c.p. module map.
Since $p^\perp u=up^\perp=u$ and $pu=up=0$, we get
$$
u^{\frac{1}{2}}\varphi_2(a)u^{\frac{1}{2}}=u^{\frac{1}{2}}(\varphi_1(a)+\omega(a)p)u^{\frac{1}{2}}=p^\perp\tilde\varphi(a)p^\perp+0=\tilde\varphi(a),
$$
therefore,
$$
u^{\frac{1}{2}}\varphi(a)u^{\frac{1}{2}}=u^{\frac{1}{2}}\mathbb E_n(\varphi_2(a))u^{\frac{1}{2}}=\mathbb E_n(u^{\frac{1}{2}}\varphi_2(a)u^{\frac{1}{2}})=\mathbb E_n(\tilde\varphi(a))=\tilde\varphi(a),
$$
for $a\in A$.
\end{proof}

\begin{proposition} \label{injc}
If $\mathfrak A$ is a unital $C^*$-algebra, $E$ is a unital injective $C^*$-algebra and a unital and central left $\mathfrak A$-module with compatible actions, $\theta: A\to B$ is an $E$-nuclear module map, then there are u.c.p. module maps $\varphi_n: \tilde A\to \mathbb M_{k(n)}(E)$ and $\psi_n: \mathbb M_{k(n)}(E)\to B$ such that
$\psi_n\circ \varphi_n\to \tilde\theta$ in the point-norm topology.
\end{proposition}
\begin{proof}
Let $\tilde\varphi_n: A\to\M$ and $\tilde\psi_n:\M\to B$ be  c.c.p. module maps with $\tilde\psi_n\circ \tilde\varphi_n\to \theta$ in the point-norm topology. By Lemma \ref{in}, there are u.c.p. module maps $\varphi_n: A\to \M$ such that
$$\tilde\varphi_n(a)=\tilde\varphi_n(1_A)^{\frac{1}{2}}\varphi_n(a)\tilde\varphi_n(1_A)^{\frac{1}{2}},$$
for $a\in A$. For $x\in\M$, let
$$\psi_n(x)=\big(\tilde\psi_n (\tilde\varphi_n(1_A))\big)^{-\frac{1}{2}}\tilde\psi_n\big(\tilde\varphi_n(1_A)^{\frac{1}{2}}x\tilde\varphi_n(1_A)^{\frac{1}{2}}\big)\big(\tilde\psi_n( \tilde\varphi_n(1_A))\big)^{-\frac{1}{2}},$$
then, as in the proof of the above lemma, one may check that $\psi_n$ is a u.c.p. module map, and clearly $\psi_n\circ \varphi_n$ converges to  $\tilde\theta$ in point-norm topology.
\end{proof}

\section{dilation and extension}

In this section we prove analogs of Wittstock and Arveson extension theorems, as well as the Stinespring dilation theorem  for module maps. Our proof of the Wittstock extension theorem follows the simple argument due to C.-Y. Suen \cite{s}.

We use a bimodule version of the notations in the previous section. Let $A$ be $C^*$-algebra, a left  $\mathfrak A$-module and a right $\mathfrak B$-module with compatible left and right actions (as in the previous section). We write $\mathfrak A\cdot A\cdot \mathfrak B$ for the closed linear span of the set of elements of the form $\alpha\cdot a\cdot \beta$, with $\alpha\in\mathfrak A, \beta \in \mathfrak B, a\in A$. We say that $A$ is neo-unital if $\mathfrak A\cdot A\cdot \mathfrak B=A$. In this case, if $\mathfrak A$ and $\mathfrak B$ are unital $C^*$-algebras, then $1_{\mathfrak A}\cdot a=a\cdot 1_{\mathfrak A}=a$, for each $a\in A$. When $\mathfrak A$ is a unital $C^*$-algebra, we say that $A$ is unital (as an  $\mathfrak A$-$\mathfrak B$-bimodule), if $A$ is a unital $C^*$-algebra and a neo-unital $\mathfrak A$-$\mathfrak B$-bimodule (we simply say, $A$ is a unital bimodule).

The next lemma extends \cite[Lemma 2.3]{s}.

\begin{lemma} \label{inj}
Let $\mathfrak A$ and $\mathfrak B$ are unital $C^*$-algebras and $A$ and $B$ be unital bimodules. Let $\theta_i: A \to B$ be a linear $\mathfrak A$-$\mathfrak B$-bimodule maps
for $i = 1,2,3,4$.

Define $\Theta: A\otimes \mathbb M_2(\mathbb C)\to  B\otimes \mathbb M_2(\mathbb C)$ by
$$\Theta\left[
                                                                       \begin{array}{cc}
                                                                         a & b \\
                                                                         c & d \\
                                                                       \end{array}
                                                                     \right]
     =\left[
     \begin{array}{cc}
                                                                         \theta_1(a) & \theta_2(b) \\
                                                                         \theta_3(c) & \theta_4(d) \\
                                                                       \end{array}
                                                                     \right].$$
If $\Theta$ is completely positive, then:

$(i)$ \ $\Theta$ is a $(\mathfrak A\oplus \mathfrak A)$-$(\mathfrak B\oplus \mathfrak B)$-bimodule map;

$(ii)$ \ $\theta_1$ and $\theta_4$  are c.p. $\mathfrak A$-$\mathfrak B$-bimodule maps;

$(iii)$ \  $\theta_2=\theta_3^*$ is a c.b. $\mathfrak A$-$\mathfrak B$-bimodule map;

$(iv)$ \  $±Re (\lambda\theta_2)< \frac{1}{2}(\theta_1+\theta_4),$ for all complex numbers $\lambda$  with $|\lambda|=1$;

$(v)$ \ the $n$-th amplifications $\theta_{in}=(\theta_i)_n$ satisfy
$$\|\Theta\|_{cb}\theta_{in}(xx^*)\geq\theta_{2n}(x)\theta_{2n}(x)^{*},$$
for $n\geq 1$, $i=1,4$, and $x\in A\otimes \mathbb M_n(\mathbb C)$.
\end{lemma}
\begin{proof}
For part $(i)$, one just needs to  consider $A\otimes \mathbb M_2(\mathbb C)$ as an $(\mathfrak A\oplus \mathfrak A)$-$(\mathfrak B\oplus \mathfrak B)$-bimodule vis the identifications
\[\alpha_1\oplus \alpha_2=\left[
                                                                       \begin{array}{cc}
                                                                         \alpha_1 & 0 \\
                                                                         0 & \alpha_2 \\
                                                                       \end{array}
                                                                     \right], \ \  \beta_1\oplus \beta_2=\left[
                                                                       \begin{array}{cc}
                                                                         \beta_1 & 0 \\
                                                                         0 & \beta_2 \\
                                                                       \end{array}
                                                                     \right],
                                                                     \]
for $ \alpha_i\in\mathfrak A,   \beta_i\in\mathfrak B; \ i=1,2.$ For part $(ii)$, let
\[\bar a= \left[
                                                                       \begin{array}{cc}
                                                                         a & a \\
                                                                         a & a \\
                                                                       \end{array}
                                                                     \right],\ \
E^{11}_{\mathfrak A}=\left[
                                                                       \begin{array}{cc}
                                                                         1_{\mathfrak A} & 0 \\
                                                                         0 & 0 \\
                                                                       \end{array}
                                                                     \right],
                                                                     \]
define the other matrix units in $\mathfrak A\otimes \mathbb M_2(\mathbb C)$ similarly (and the same for $\mathfrak B$), and observe that
\[E^{11}_{\mathfrak A}\cdot\Theta(\bar a)\cdot E^{11}_{\mathfrak B}=\left[
                                                                       \begin{array}{cc}
                                                                         \theta_1(a) & 0 \\
                                                                         0 & 0 \\
                                                                       \end{array}
                                                                     \right],\ \
E^{22}_{\mathfrak A}\cdot\Theta(\bar a)\cdot E^{22}_{\mathfrak B}=\left[
                                                                       \begin{array}{cc}
                                                                         0 & 0 \\
                                                                         0 & \theta_4(a) \\
                                                                       \end{array}
                                                                     \right].
\]
Part $(iii)$ follows from $\Theta(\bar a^*)=\Theta(\bar a)^*$ and the fact that $\|\theta_2\|_{cb}\leq \|\Theta(1_{A\otimes \mathbb M_2(\mathbb C)})\|$. Part $(iv)$ follows from the fact that the element $$(E^{11}_{\mathfrak A}+E^{12}_{\mathfrak A})(\pm\lambda E^{11}_{\mathfrak A}+E^{22}_{\mathfrak A})\Theta(\bar a)(\pm\bar\lambda E^{11}_{\mathfrak B}+E^{22}_{\mathfrak B})(E^{11}_{\mathfrak B}+E^{21}_{\mathfrak A})$$
  is positive in $B\otimes \mathbb M_2(\mathbb C)$. Finally, part $(v)$ is a direct consequence of  the Schwartz inequality for c.p.  maps.                                                           \end{proof}

Next we extend the Wittstock theorem for completely positive maps into injective $C^*$-algebras \cite[Satz 4.5]{w1} (here, we adapt the proof of \cite[Theorem 2.4]{s}). We need the following extension of \cite[Satz 4.2]{w1}. Let $A, B$ be as in the above lemma, following Wittstock \cite[Definition 4.1]{w1}, we say that a linear hermitian map $\theta: A\to B$ is $\mathfrak A$-$\mathfrak B$-matricially bounded if there is $k>0$ such that for each $n\geq 1$, self-adjoint elements $\alpha\in\mathbb M_n(\mathfrak A), \beta\in\mathbb M_n(\mathfrak B)$ and $a\in\mathbb M_n(A)$ with $\pm a\leq \alpha\cdot 1_{A}\cdot \beta$ we have $\pm\theta_n(a)\leq k(\alpha\cdot 1_{B}\cdot \beta)$.

\begin{lemma} \label{hom}
Let $\mathfrak A$ and $\mathfrak B$ are unital $C^*$-algebras and $A$ and $B$ be unital bimodules. For a  linear hermitian  map $\theta: A\to B$, the following are equivalent:

 $(i)$ \ $\theta$ is $\mathfrak A$-$\mathfrak B$-matricially bounded,

 $(ii)$ \ $\theta$ is $\mathfrak A$-$\mathfrak B$-bimodule map.
\end{lemma}
\begin{proof}
If $(i)$ holds and $a\in A$, $\|a\|=1$, then multiplying from both sides of
\[\pm\left[
                                                                       \begin{array}{cc}
                                                                         0 & a \\
                                                                         a^* & 0 \\
                                                                       \end{array}
                                                                     \right]\leq
\left[
                                                                       \begin{array}{cc}
                                                                         1_A & 0 \\
                                                                         0 & 1_A \\
                                                                       \end{array}
                                                                     \right]
\]
by the self-adjoint element $\left[
                                                                       \begin{array}{cc}
                                                                         p\cdot 1_A & 0 \\
                                                                         0 & 1_A\cdot q \\
                                                                       \end{array}
                                                                     \right]$ for projections $p\in\mathfrak A$ and $q\in\mathfrak B$, using the compatibility conditions, we get
\[\pm\left[
                                                                       \begin{array}{cc}
                                                                         0 & p\cdot a\cdot q \\
                                                                         p\cdot a^*\cdot q & 0 \\
                                                                       \end{array}
                                                                     \right]\leq
\left[
                                                                       \begin{array}{cc}
                                                                         p\cdot 1_A & 0 \\
                                                                         0 & 1_A\cdot q \\
                                                                       \end{array}
                                                                     \right],
\]
thus
\[\pm\left[
                                                                       \begin{array}{cc}
                                                                         0 & \theta(p\cdot a\cdot q) \\
                                                                         \theta(p\cdot a^*\cdot q) & 0 \\
                                                                       \end{array}
                                                                     \right]\leq
k\left[
                                                                       \begin{array}{cc}
                                                                         p\cdot 1_B & 0 \\
                                                                         0 & 1_B\cdot q \\
                                                                       \end{array}
                                                                     \right].
\]
Multiplying from both sides by the self-adjoint element $\left[
                                                                       \begin{array}{cc}
                                                                         (1-p)\cdot 1_A & 0 \\
                                                                         0 & 1_A \\
                                                                       \end{array}
                                                                     \right]$
we get
\[\pm\left[
                                                                       \begin{array}{cc}
                                                                         0 & (1-p)\cdot\theta(p\cdot a\cdot q) \\
                                                                         (1-p)\cdot\theta(p\cdot a^*\cdot q) & 0 \\
                                                                       \end{array}
                                                                     \right]\leq
k\left[
                                                                       \begin{array}{cc}
                                                                         0 & 0 \\
                                                                         0 & 1_B\cdot q \\
                                                                       \end{array}
                                                                     \right].
\]
Hence $(1-p)\cdot\theta(p\cdot a\cdot q)=0$. Similarly, $p\cdot\theta((1-p)\cdot a\cdot q)=0$. By the same argument, $\theta(p\cdot a\cdot q)\cdot (1-q)=0$ and  $\theta(p\cdot a\cdot(1-q))\cdot q=0$. Thus
$$\theta(p\cdot a\cdot q)=p\cdot\theta(p\cdot a\cdot q)+(1-p)\cdot\theta(p\cdot a\cdot q)=p\cdot\theta(p\cdot a\cdot q)+p\cdot\theta((1-p)\cdot a\cdot q)=p\cdot\theta(a\cdot q),$$
and similarly, $\theta(p\cdot a\cdot q)=\theta(p\cdot a)\cdot q$. This establishes $(i)$ for the case that all the $C^*$-algebras involved are von Neumann algebras. Now by the kaplansky density theorem, $\theta^{**}: A^{**}\to B^{**}$ is $\mathfrak A^{**}$-$\mathfrak B^{**}$-matricially bounded, and so, by the above argument, it is an $\mathfrak A^{**}$-$\mathfrak B^{**}$-bimodule map, and $(ii)$ follows.

Conversely, if $(ii)$ holds, given $\varepsilon>0$, $n\geq 1$, self-adjoint elements $\alpha\in\mathbb M_n(\mathfrak A), \beta\in\mathbb M_n(\mathfrak B)$ and $a\in\mathbb M_n(A)$ with $\pm a\leq \alpha\cdot 1_{A}\cdot \beta$, put $a_{\varepsilon}:=(\alpha+\varepsilon 1_{\mathfrak A})\cdot a\cdot (\beta+\varepsilon 1_{\mathfrak B})$. Then $\|a_{\varepsilon}\|=1$ and $\|\theta(a_{\varepsilon})\|\leq\|\theta\|_{cb}$, hence $\pm\theta(a)\leq \|\theta\|_{cb}(\alpha\cdot 1_{B}\cdot \beta)$.
\end{proof}

\begin{theorem}[Wittstock extension theorem] \label{injc}
Let $\mathfrak A$ and $\mathfrak B$ are unital $C^*$-algebras and $A$ and $B$ be unital bimodules. Let $B$ be injective as a $C^*$-algebra. Then for each $\mathfrak A$-$\mathfrak B$-subbimodule $A_0$ of $A$, each c.b.
bimodule map $\theta_0: A_0\to B$  has a c.b. bimodule  extension $\theta: A \to B$ with the same $cb$-norm.
\end{theorem}
\begin{proof}
Assume that $\|\theta_0\|_{cb}=1$ and put
\[\mathcal S=\{\left[
                                                                       \begin{array}{cc}
                                                                         \alpha & a \\
                                                                         b^* & \beta \\
                                                                       \end{array}
                                                                     \right]:\ \ \alpha\in\mathfrak A, \beta\in \mathfrak B, a,b\in A_0\}. \]
Define $\Theta_0: \mathcal S\to B\otimes \mathbb M_2(\mathbb C)$ by
$$\Theta_0\left[
                                                                       \begin{array}{cc}
                                                                         \alpha & a \\
                                                                         b^* & \beta \\
                                                                       \end{array}
                                                                     \right]
     =\left[
     \begin{array}{cc}
                                                                         \alpha & \theta_0(a) \\
                                                                         \theta_0(b)^* & \beta \\
                                                                       \end{array}
                                                                     \right].$$
For $n\geq 1, \varepsilon >0$ and positive element $\left[
                                                                       \begin{array}{cc}
                                                                         \bar\alpha & \bar a \\
                                                                         \bar a^* & \bar\beta \\
                                                                       \end{array}
                                                                     \right]\in \mathcal S\otimes \mathbb M_n(\mathbb C)$, we have
 \[
 \left[\begin{array}{cc}
                                                                         (\bar\alpha+\varepsilon)^{-\frac{1}{2}} & 0 \\
                                                                         0 & (\bar\beta+\varepsilon)^{-\frac{1}{2}} \\
                                                                       \end{array}
                                                                     \right]\cdot
\left[
                                                                       \begin{array}{cc}
                                                                         \bar\alpha+\varepsilon & \bar a \\
                                                                         \bar a^* & \bar\beta+\varepsilon \\
                                                                       \end{array}
                                                                     \right]\cdot
 \left[\begin{array}{cc}
                                                                         (\bar\alpha+\varepsilon)^{-\frac{1}{2}} & 0 \\
                                                                         0 & (\bar\beta+\varepsilon)^{-\frac{1}{2}} \\
                                                                       \end{array}
                                                                     \right]\geq 0,
\]
in $\mathcal S\otimes \mathbb M_n(\mathbb C)$. By \cite[page 162]{ce},
\[\|\theta_{0n}\big((\bar\alpha+\varepsilon)^{-\frac{1}{2}} \bar a(\bar\beta+\varepsilon)^{-\frac{1}{2}}\big)\|\leq\|(\bar\alpha+\varepsilon)^{-\frac{1}{2}} \bar a(\bar\beta+\varepsilon)^{-\frac{1}{2}}\|\leq 1,\]
thus
$\left[
                                                                       \begin{array}{cc}
                                                                         \bar\alpha & \bar a \\
                                                                         \bar a^* & \bar\beta \\
                                                                       \end{array}
                                                                     \right]$
is positive in $\mathcal S\otimes \mathbb M_n(\mathbb C)$. Therefore, $\Theta_0$ is c.p. and so extends to a u.c.p. map $\Theta: A\otimes\mathbb M_2(\mathbb C)\to B\otimes \mathbb M_2(\mathbb C)$. Let us show that $\Theta$ is $(\mathfrak A\oplus\mathfrak A)$-$(\mathfrak B\oplus\mathfrak B)$-matricially bounded. Given $n\geq 1$, self-adjoint elements $\alpha\in\mathbb M_n(\mathfrak A\oplus\mathfrak A), \beta\in\mathbb M_n(\mathfrak B\oplus\mathfrak B)$ and $a\in A\otimes\mathbb M_2(\mathbb C)\otimes\mathbb M_n(\mathbb C)$ with $\pm a\leq \alpha\cdot 1_{A}\cdot \beta$, we have $\Theta_n(\alpha\cdot 1_{A\otimes\mathbb M_2(\mathbb C)}\cdot \beta\pm a)\geq 0$. On the other hand,
\[\Theta_n(\alpha\cdot 1_{A\otimes\mathbb M_2(\mathbb C)}\cdot \beta)
=\left[
     \begin{array}{cc}
                                                                         \alpha & \theta_0(0) \\
                                                                         \theta_0(0)^* & \beta \\
                                                                       \end{array}
                                                                     \right]=\alpha\cdot
\left[
                                                                     \begin{array}{cc}
                                                                         1_B & 0 \\
                                                                         0 & 1_B \\
                                                                       \end{array}
                                                                     \right]\cdot\beta=\alpha\cdot 1_{B\otimes\mathbb M_2(\mathbb C)}\cdot \beta,
\]
hence, $\pm\Theta_n(a)\leq \alpha\cdot 1_{B\otimes\mathbb M_2(\mathbb C)}\cdot \beta$. Therefore, by Lemma \ref{hom}, $\Theta$ is an $(\mathfrak A\oplus\mathfrak A)$-$(\mathfrak B\oplus\mathfrak B)$-bimodule map. Define $\theta: A\to B$ by
\[\Theta\left[
                                                                       \begin{array}{cc}
                                                                         0 & a \\
                                                                         0 & 0 \\
                                                                       \end{array}
                                                                     \right]=\left[
                                                                       \begin{array}{cc}
                                                                         * & \theta(a) \\
                                                                         * & * \\
                                                                       \end{array}
                                                                     \right].
                                                                     \]
This is clearly an extension of $\theta_0$ to an $\mathfrak A$-$\mathfrak B$-bimodule map with the same $cb$-norm.
\end{proof}

To state and prove the module version of Stinespring dilation theorem we need the notion, adapted from  $W^*$-correspondence, introduced by Alain Connes \cite{c} and developed by Sorin Popa \cite{po}. For $C^*$-algebras $\mathfrak A$ and $\mathfrak B$, a $\mathfrak A$-$\mathfrak B$-Hilbert space is a Hilbert space $H$ with two representations $\pi_{\mathfrak A}: \mathfrak A\to \mathbb B(H)$ and $\pi_{\mathfrak B^{op}}: \mathfrak B^{op}\to \mathbb B(H)$ with commuting ranges, where  $\mathfrak B^{op}$ is the opposite algebra of $\mathfrak B$. We write $\alpha\cdot\xi\cdot\beta$ for $\pi_{\mathfrak A}(\alpha)\pi_{\mathfrak B^{op}}(\beta)\xi$, for $\alpha\in \mathfrak A$, $\beta\in \mathfrak B$ and $\xi\in H$. A morphism in the category of $\mathfrak A$-$\mathfrak B$-Hilbert spaces is a bounded linear bimodule map.

If $H$ is a $\mathfrak B$-$\mathfrak A$-Hilbert space, $\mathbb B(H)$ is a $\mathfrak A$-$\mathfrak B$ $C^*$-module in our sense, with module actions $\alpha\cdot T\cdot\beta(\xi)=T(\beta\cdot\xi\cdot\alpha),$ for $T\in \mathbb B(H)$.

\begin{theorem}[Stinespring dilation theorem] \label{dial}
Let $\mathfrak A$ and $\mathfrak B$ are unital $C^*$-algebras and $A$  be a unital $\mathfrak A$-$\mathfrak B$-bimodule. Let $\mathcal H$ be a $\mathfrak B$-$\mathfrak A$-Hilbert space and $\theta: A\to  \mathbb B(\mathcal H)$ be a c.p. bimodule map. Then there is a $\mathfrak B$-$\mathfrak A$-Hilbert space $\mathcal K$ and a representation $\pi:A\to  \mathbb B(\mathcal K)$, which is also an $\mathfrak A$-$\mathfrak B$-bimodule map, and  a bounded linear bimodule map $V: \mathcal H\to\mathcal K$  such that $\theta(a)=V^*\pi(a)V$, for $a\in A$.
\end{theorem}
\begin{proof}
It is clear  that $\hat{\mathcal H}:=A\otimes_{\mathfrak B}\mathcal H$ is an $\mathfrak A$-$\mathfrak A$-Hilbert space. Note that the left and right actions of $\mathfrak A$ on $\hat{\mathcal H}$ commute. By the classical Stinespring dilation theorem, there is a  representation $\hat\pi:A\to  \mathbb B(\hat{\mathcal H})$ and   bounded linear map $W: \mathcal H\to\hat{\mathcal H}; \xi\mapsto 1_A\otimes \xi,$  such that $\theta(a)=W^*\hat\pi(a)W$, for $a\in A$. Then $\hat\pi^{op}:A^{op}\to  \mathbb B(\hat{\mathcal H})$ is a c.p. map, so we may form the $\mathfrak B$-$\mathfrak A$-Hilbert space $\mathcal K=A^{op}\otimes_{\mathfrak A}\hat{\mathcal H}$. Let us define $\pi: A\to  \mathbb B(\mathcal K)$ by
$$\pi(a)(\sum_i b_i\otimes \eta_i)=\sum_i ab_i\otimes \eta_i,$$
for $a\in A$ and finite sums in the algebraic tensor product. Then for the  bounded linear map $V: \hat{\mathcal H}\to \mathcal K; \eta\mapsto 1_A\otimes \eta,$ we have, $\hat\pi(a)=V^*\pi(a)V$, and so $\theta(a)=(WV)^*\hat\pi(a)WV$, for $a\in A$. We only need to check that $\pi$ is a bimodule map. We have,
\begin{align*}
\pi(\alpha\cdot a)(\sum_i b_i\otimes \eta_i)&=\sum_i (\alpha\cdot a)b_i\otimes \eta_i=\sum_i \alpha\cdot ab_i\otimes \eta_i\\
&=\sum_i  ab_i\cdot_{op}\alpha\otimes \eta_i=\sum_i  ab_i\otimes \alpha\cdot\eta_i\\
&=\sum_i  ab_i\otimes\eta_i\cdot \alpha=(\sum_i  ab_i\otimes\eta_i)\cdot \alpha\\
&=\pi(a)\big((\sum_i b_i\otimes \eta_i)\cdot\alpha\big)\\&=(\alpha\cdot\pi(a))(\sum_i b_i\otimes \eta_i),
\end{align*}
and
\begin{align*}
\pi( a\cdot\beta)(\sum_i b_i\otimes \eta_i)&=\sum_i (a\cdot\beta)b_i\otimes \eta_i=\sum_i (\beta\cdot_{op}a)b_i\otimes \eta_i\\
&=\sum_i \beta\cdot_{op}(ab_i)\otimes \eta_i=\sum_i  (ab_i)\cdot\beta\otimes \eta_i\\
&=\pi(a)(\sum_i b_i\cdot\beta\otimes \eta_i)=\pi(a)(\sum_i \beta\cdot_{op}b_i\otimes \eta_i)\\
&=\pi(a)\big(\beta\cdot_{op}(\sum_i \beta\cdot_{op}b_i\otimes \eta_i)\big)\\&=(\pi(a)\cdot\beta)(\sum_i \beta\cdot_{op}b_i\otimes \eta_i),
\end{align*}
as claimed.
\end{proof}

\begin{remark}
$(i)$ In the non unital case, we could get the same result under the assumption that $\theta$ is {\it strict}, that is, $\{\theta(e_i)\}$ is strictly Cauchy in $\mathbb B(H)$ for some bounded approximate identity $(e_i)$ of $A$. This is weaker than non degeneracy and is automatic in the unital case, or when $\mathfrak A=\mathbb C$ or $\mathfrak B=\mathbb C$ (i.e., $\mathcal H$ is a Hilbert space). This could be proved as in \cite[Theorem 5.6]{lan} (with slight modifications, as in \cite{lan} the c.p. map is not assumed to be a module map.)

$(ii)$ When the dilation is non degenerate (or so called minimal), i,e., $\pi(A)V\mathcal H$ is dense in $\mathcal K$, and $\theta$ is a c.c.p. module map, we may add an ingredient $\rho: \theta(A)^{'}\to \pi(A)^{'}$ such that $\theta(a)x=V^*\pi(a)\rho(x)V$, for $a\in A$ and $x\in \theta(A)^{'}$. The module map $\rho$ is defined on the dense subspace $\pi(A)V\mathcal H$ by
$$\rho(x)(\sum_i\pi(a_i)V\xi_i)= \sum_i\pi(a_i)Vx\xi_i,$$
and the rest goes as in the classical case \cite[1.5.6]{bo}.
\end{remark}

To prove a module version of Arveson extension theorem, we need some preparation. We return to the setup of the previous section and consider left $\mathfrak A$-modules again. For left $C^*$-modules $A$ and $B$, let $CP_{\mathfrak A}(A,B)$ and $P_{\mathfrak A}(A,B)$ denote the set of c.p. and positive left module linear maps from $A$ to $B$, respectively.

\begin{lemma} \label{cp1}
Let $\mathfrak A$ be a unital $C^*$-algebra and $A$ be a unital module. For each $n\geq 1$,

$(i)$ $CP_{\mathfrak A}(\mathbb M_n(\mathfrak A), A)=CP_{\mathbb C}(\mathbb M_n(\mathbb C), A)$,

$(ii)$ $CP_{\mathfrak A}(\mathbb M_n(\mathfrak A), A)=\mathbb M_n(A)_{+}.$
\end{lemma}
\begin{proof}
$(i)$ Recall that a set of matrix units in a unital $C^*$-algebra is a set $\{e_{ij}\}$ of elements satisfying
$$e_{ij}^*=e_{ji}, \ e_{ij}e_{kl}=\delta_{jk}e_{il}, \ \sum_i e_{ii}=1.$$
If $e_{ij}\in \mathbb M_n(\mathbb C)$ is the standard set of matrix units, to each $\theta \in CP_{\mathbb C}(\mathbb M_n(\mathbb C), A)$ we associate $\tilde\theta: \mathbb M_n(\mathfrak A) \to A$ defined by $$\tilde\theta(\sum_{i,j=1}^{n} \alpha_{ij}\otimes e_{ij})=\sum_{i,j=1}^{n} \alpha_{ij}\cdot \theta(e_{ij}).$$
This is clearly a module map. To see it is c.p., observe that $\tilde\theta$ is a composition of the c.p. map $id\otimes\theta: \mathfrak A\otimes^{h}\mathbb M_n(\mathbb C)\to \mathfrak A\otimes^{h}A$ with the c.p. module action $\cdot: \mathfrak A\otimes^{h}A\to A$, thus $\tilde\theta\in CP_{\mathfrak A}(\mathbb M_n(\mathfrak A), A).$

To go the other direction, we to any $\sigma\in CP_{\mathfrak A}(\mathbb M_n(\mathfrak A), A)$  the map $\theta: \mathbb M_n(\mathbb C)\to A$ defined by $\theta(x)=\sigma(1_{\mathfrak A}\otimes x)$, which is clearly a c.p. map and, since $\sigma$ is a module map, $$\tilde\theta(\sum_{i,j=1}^{n} \alpha_{ij}\otimes e_{ij})=\sum_{i,j=1}^{n}\sigma(1_{\mathfrak A}\otimes e_{ij})=\sigma(\sum_{i,j=1}^{n} \alpha_{ij}\otimes e_{ij}),$$ i.e., $\tilde\theta=\sigma$.

$(ii)$ This follows from $(i)$ and the fact that $CP_{\mathbb C}(\mathbb M_n(\mathbb C), A)=\mathbb M_n(A)_{+},$ by \cite[1.5.12]{bo}.
\end{proof}

We have the following extension of \cite[Proposition 1.5.14]{bo}. The idea of the proof is similar, but of course the GNS-construction used in \cite{bo} is not working in this general setting (and is replaced by an application of the above Stinespring dilation theorem).

\begin{lemma} \label{cp2}
Let $\mathfrak A$ be a unital  $C^*$-algebra and $A$ is a unital module. For each $n\geq 1$, then $CP_{\mathfrak A}(A, \mathbb M_n(\mathfrak A))=CP_{\mathfrak A}(\mathbb M_n(A), \mathfrak A).$
\end{lemma}
\begin{proof}
To each $\varphi\in CP_{\mathfrak A}(A, \mathbb M_n(\mathfrak A))$ we are going to associate $\hat\varphi\in CP_{\mathfrak A}(\mathbb M_n(A), \mathfrak A),$ via
$$\hat\varphi([a_{ij}])=\sum_{i,j=1}^{n}\varphi(a_{ij})_{ij}.$$
For $k,n\geq 1$, let $\{e_i\}_{1}^{n}$ be an ONB of $\ell_n^2$. Let $e=[e_1,\cdots,e_k]^T\in (\ell_n^2)^k$. Given $\varphi\in CP_{\mathfrak A}(A, \mathbb M_n(\mathfrak A))$, let $\varphi_k$ be its amplification on $\mathbb M_k(A))$. Then, for the case $k=n$, we have
\begin{align*}
\hat\varphi([a_{ij}])&=\sum_{i,j=1}^{n}\varphi(a_{ij})_{ij}\\
&=\sum_{i,j=1}^{n}\langle\varphi(a_{ij})(e_{j}\otimes 1_{\mathfrak A}), (e_{i}\otimes 1_{\mathfrak A})\rangle \\
&=\langle\varphi_n([a_{ij}])(e\otimes 1_{\mathfrak A}), (e\otimes 1_{\mathfrak A})\rangle\geq 0,
\end{align*}
thus $\hat\varphi$ is positive. Since the above correspondence behaves canonically on amplifications, a similar argument shows that each amplification $\hat\varphi_k$ is positive, i.e., $\hat\varphi$ is c.p.

Conversely, if $\hat\varphi$ is c.p., take a faithful  representation of $\mathfrak A\subseteq B(H)$. Consider $H$ as a $\mathbb C$-$\mathfrak A$-Hilbert space and apply Theorem \ref{dial} to $\hat\varphi\in CP_{\mathfrak A}(\mathbb M_n(A), B(H))$ to get a $\mathfrak A$-$\mathbb C$-Hilbert space $K$,  a representation $\pi: \mathbb M_n(A)\to B(K)$ and a bounded bimodule map $V: H\to K$ such that
$$\hat\varphi([a_{ij}])=V^*\pi([a_{ij}])V.$$
Let $\{e_{ij}\}$ be the set of canonical matrix units of $\mathbb M_n(\mathbb C)$, then  for $a\in A$, $e_{ij}\otimes a$ could be regarded as an element in $\mathbb M_n(A)$, and
\begin{align*}
\varphi(a)&=[\varphi(a)_{ij}]=[\hat\varphi(e_{ij}\otimes a)]\\
&=[V^*\pi(e_{ij}\otimes a)V]\\&=(V^n)^*[\pi(e_{ij}\otimes a)]V^n\\
&=(V^n)^*\pi_n[(e_{ij}\otimes a)]V^n.
\end{align*}
Hence $\varphi$ is positive. Again by repeating the same calculations for amplifications, we get that $\varphi$ is also c.p.
\end{proof}

\begin{lemma} \label{cp3}
Let $\mathfrak A$ be a unital  $C^*$-algebra and $A$ is a unital module. Let $E\subseteq A$ be an operator subsystem and a submodule. Then

$(i)$ If $\psi: E\to \mathfrak A$ is a positive module map, then every norm-preserving extension of $\psi$ to a module map on $A$ is positive.

$(ii)$ If $\psi: E\to \mathbb M_n(\mathfrak A)$ is a c.p. module map, then there is a norm-preserving c.p. extension of $\psi$ to a module map on $A$.
\end{lemma}
\begin{proof}
$(i)$\ Given $0<\varepsilon<\|\psi\|$, choose $x\in E$ with $\|x\|\leq 1$ and $\|\psi(x)\|\geq \|\psi\|-\varepsilon$. If $\alpha:=\psi(x)$, then replacing $x$ by $\frac{\alpha}{\|\alpha\|}\cdot x$, we may assume that $\psi(x)\geq 0$. Again, replacing $x$ by $\frac{1}{2}(x+x^*)$, we may assume that $x$ is self-adjoint. Thus $x\leq \|x\|1_A$, and so $\psi(x)\leq\|x\|\psi(1_A)$. Since both $\psi(x)$ and $\psi(1_A)$ are positive,
$$\|\psi\|-\varepsilon\leq\|\psi(x)\|\leq\|x\|\ \|\psi(1_A)\|\leq\|\psi(1_A)\|\leq\|\psi\|,$$
thus $\|\psi(1_A)\|=\|\psi\|.$

Let $\tilde\psi$ be any norm-preserving extension of $\psi$ to a module map on $A$. For simplicity, let us assume that $\|\tilde\psi\|=\|\psi\|=1$. Let $a\in A$ be a positive element with $\|a\|\leq 1$ and put $\tilde\psi(a)=\beta+i\gamma$, with $\beta,\gamma$ self adjoint in $\mathfrak A$ and $\gamma\neq 0$. By Hahn Banach, choose a linear functional $\phi$ on $\mathfrak A$ with $\phi(\psi(1_A))=1=\|\phi\|$ and $\phi(\gamma)\neq 0$. Then $\phi$ is a state and
$$ \phi(\tilde\psi(1_A))=\|\psi(1_A)\|=1=\|\tilde\psi\|\geq\|\phi\circ\tilde\psi\|,$$
hence $\phi(\tilde\psi(1_A))=\|\phi\circ\tilde\psi\|$, and so $\phi\circ\tilde\psi$ is a positive functional on $A$. On the other hand,  $0\leq\phi\circ\tilde\psi(a)=\phi(\beta)+i\phi(\gamma)$, and since positive functionals are self adjoint, we get $\phi(\gamma)=0$, which is a contradiction. Thus $\gamma=0$, that is, $\beta=\tilde\psi(a)$ is self adjoint.

Next, $0\leq\psi(1_A)=\tilde\psi(1_A)\leq 1_{\mathfrak A}$, hence $\beta-1_{\mathfrak A}\leq \tilde\psi(a-1_A)$ and both sides are self adjoint, thus $$\|\beta-1_{\mathfrak A}\|\leq \|\tilde\psi(a-1_A)\|\leq \|a-1_A\|\leq 1,$$
therefore, $\beta\geq 0$ in $\mathfrak A$.

$(ii)$\ To $\psi$ we associate a c.p. module map $\hat\psi: \mathbb M_n(E)\to \mathfrak A$ via Lemma \ref{cp2}, whose extension to $\mathbb M_n(E)$ is again c.p. by this lemma and part $(i)$. This proves $(ii)$ using the above lemma again.
\end{proof}

Now we are ready to prove the Arveson extension theorem for module maps.

\begin{theorem} [Arveson extension theorem] \label{arv}
Let $\mathfrak A$ be a unital  $C^*$-algebra and $A$ is a unital $($left$)$ module. Let $E\subseteq A$ be an operator subsystem and a submodule. Let $H$ be a right Hilbert $\mathfrak A$-module with a standard frame such that $\mathbb B(H)$ is a von Neumann algebra. Then each c.c.p. module map $\varphi: E\to \mathbb B(H)$ has a c.c.p. module map extension $\tilde\varphi: A\to \mathbb B(H)$.
\end{theorem}
\begin{proof}
Let $K$ be a finitely generated submodule of $H$. Then $K$ is the range of a finite rank adjointable operator, and so is complemented in $H$ \cite[Theorem 3.2]{lan}. hence $K$ is the range of a finite rank orthogonal projection. Let $\mathfrak B=\{x_\lambda\}$ be a standard frame for $H$, then for each $x\in H$, $x=\sum_\lambda\langle x,x_\lambda\rangle x_\lambda$, with convergence in norm. For each finite subset $i\subseteq \mathfrak B$, let $p_i$ be the corresponding orthogonal projection whose range is generated by $i$, then by the above norm convergence, $\{p_i\}$ is a decreasing net of finite rank projections converging to 1 in the strong operator topology in $\mathbb B(H)$.

If the range of $p_i$ is generated by $\xi_1,\cdots,\xi_{n_i}$ and $e_1,\cdots,e_{n_i}$ is the standard basis of $\ell^2_{n_i}$ then $p_iH$ could be identified with $\ell^2_{n_i}\otimes \mathfrak A$ via $\xi_j\mapsto e_j\otimes 1_{\mathfrak A}$. We then use the identifications
$$p_i\mathbb B(H)p_i=\mathbb B(p_iH)=\mathbb B(\ell^2_{n_i}(\mathfrak A))=\mathbb M_{n_i}(\mathfrak A).$$
Consider the c.c.p. module map $\varphi: E\to \mathbb B(H)$ as above and define $$\varphi_i: E\to \mathbb M_{n_i}(\mathfrak A); \ x\mapsto p_i\varphi(x)p_i.$$
Then $\varphi_i$ is a c.c.p. map. Let us observe that it is also a module map. Since the range of $p_i$ is a submodule of $H$, we have
$$(\alpha\cdot p_i)(\xi)=p_i(\xi\cdot\alpha)=p_i(\xi)\cdot\alpha,$$
for each $\alpha\in \mathfrak A$ and $\xi\in H$. Therefore,
\begin{align*}
p_i\varphi(\alpha\cdot x)p_i(\xi) &=p_i\big(\alpha\cdot \varphi(x)\big)p_i(\xi)\\
&=p_i\big(\varphi(x)(p_i(\xi)\cdot \alpha)\big)\\
&=p_i\big(\varphi(x)(p_i(\xi\cdot \alpha))\big)\\
&=\big(\alpha\cdot p_i\varphi(x)p_i\big)(\xi),
\end{align*}
for each $\alpha\in \mathfrak A$ and $\xi\in H$. By Lemma \ref{cp3}$(ii)$, $\varphi_i$ has a c.c.p. module map extension $\tilde\varphi_i: A\to \mathbb M_{n_i}(\mathfrak A)$. Consider $\tilde\varphi_i$ as a map into $\mathbb B(H)$ and use \cite[1.3.7]{bo} to get a cluster point $\tilde\varphi$ of the net $\{\tilde\varphi_i\}$ in the point-ultra-weak topology of $\mathbb B(H)$. This is a c.c.p. module map extension of $\varphi$.
\end{proof}

In the above theorem, the condition that $H$ has a standard frame is known to be satisfied when $H$ is finitely or countably generated. Also the condition that $\mathbb B(H)$ is a von Neumann algebra is satisfied when $\mathfrak A$ is a von Neumann algebra and $H$ is a self dual Hilbert $\mathfrak A$-module \cite[Theorem 3.10]{pa}. Both conditions are also satisfied if $H$ is a $\mathbb C$-$\mathfrak A$-Hilbert space (i.e., $H$ is a Hilbert space in which $\mathfrak A$ is represented.)

\begin{corollary}
Let $\mathfrak A$ be a unital  $C^*$-algebra and $A$ is a $C^*$-algebra and $\mathfrak A$-module with compatible actions. Consider the faithful representation $A\subseteq \mathbb B(H)$, where  $H$ is a right Hilbert $\mathfrak A$-module with standard frame such that $\mathbb B(H)$ is a von Neumann algebra. Then $A$ is injective in the category of $C^*$-algebras with compatible $\mathfrak A$-module structure with c.c.p. module maps, if and only if there is a conditional expectation $\mathbb E: \mathbb B(H)\to A$ which is a module map. This is independent of the choice of the faithful representation.
\end{corollary}
\begin{proof}
If the conditional expectation exists, and $B\subseteq C$ is a $C^*$-subalgebra and a submodule, for a $C^*$-algebra $C$ with compatible $\mathfrak A$-module structure, and $\varphi: B\to A$ is a c.c.p. module map, then by Theorem \ref{arv}, there is a  c.c.p. module map extension $\tilde\varphi: C\to \mathbb B(H)$. Thus $\mathbb E\circ\tilde\varphi: C\to A$ is a  c.c.p. module map extension.

Conversely, if $A$ is injective, again by the above theorem, the identity map: $A\to \mathbb A$ extends to a conditional expectation: $\mathbb B(H)\to A$. The independence from the choice of the faithful representation is clear.
\end{proof}

When $\mathfrak A$ and $A$ are von Neumann algebras, we always assume that the action of $\mathfrak A$ on $A$ is normal, that is, for each $a\in A$, the map $\alpha\mapsto \alpha\cdot a$ is a normal map from $\mathfrak A$ to $A$.

\begin{lemma} \label{normal}
Let $\mathfrak A$ be a finite, atomic von Neumann algebra with no central summands of type $I_{\infty, \infty}$. Let $A$ and $B$ be von Neumann algebras and left $\mathfrak A$-modules with compatible normal actions and let $B$ be injective as a von Neumann algebra. Then for each c.p. module map $\varphi: A\to B$, there is a net $\{\varphi_i\}$ of normal c.p. module maps from $A$ to $B$, converging to $\varphi$ in the point-norm topology.
\end{lemma}
\begin{proof}
Consider faithful normal representations $A\subseteq \mathbb B(H)$ and $B\subseteq \mathbb B(K)$ on Hilbert spaces $H$ and $K$. The normal representation $\pi_A: \mathfrak A\to \mathbb B(H); \alpha\mapsto \alpha\cdot 1_A$ gives a right Hilbert $\mathfrak A$-module on $H$. Similarly, $K$ is a right Hilbert $\mathfrak A$-module. By Theorem \ref{arv}, $\varphi$ has a c.p. module map extension $\tilde\varphi: \mathbb B(H)\to B$. By \cite[Corollary 3.3]{m}, there is a net $\tilde\varphi_i: \mathbb B(H)\to \mathbb B(K)$ of c.p. normal module maps, converging to $\tilde\varphi$ in the point-norm topology. Let $\mathbb E: \mathbb B(K)\to B$ be a conditional expectation. This is a c.p. $B$-module map, and since $B$ is unital, $\mathbb E$ is also a $\mathfrak A$-module map. Let $\varphi_i$ be the restriction of $\mathbb E\circ\tilde\varphi_i$ to $A$. Then the net $\varphi_i: \mathbb B(H)\to \mathbb B(K)$, consisting of c.p. normal module maps, converges to $\varphi$ in the point-norm topology.
\end{proof}

\begin{corollary}
Let $\mathfrak A$ be a finite, atomic, injective von Neumann algebra with no central summands of type $I_{\infty, \infty}$ and $H$ be a self dual Hilbert $\mathfrak A$-module with standard frame. Let $E\subseteq \mathbb B(H)$ be an ultra-weakly closed operator subsystem and a submodule. Let $\psi: E\to \mathbb M_n(\mathfrak A)$ be a c.c.p. module map, then there is a net of ultra-weakly continuous c.c.p. module maps $\psi_i: E\to \mathbb M_n(\mathfrak A)$, converging to $\psi$ in point-norm topology.
\end{corollary}
\begin{proof}
By Theorem \ref{arv}, we may assume that $E=\mathbb B(H)$. By Lemma \ref{cp2}, we get a map $\hat\psi\in CP_{\mathfrak A}(\mathbb M_n(\mathbb B(H)), \mathfrak A).$ By Lemma \ref{normal} there is a net of normal module maps $\{\hat\psi_i\} \subseteq CP_{\mathfrak A}(\mathbb M_n(\mathbb B(H)), \mathfrak A)$, converging to $\hat\psi$ in the point-norm topology. Then the net $\{\psi_i\} \subseteq CP_{\mathfrak A}(\mathbb B(H), \mathbb M_n(\mathfrak A))$ converges to $\psi$ in the point-norm topology. Since $\|\psi_i(1)-\psi(1)\|\to 0$, as \cite[1.6.3]{bo}, we may adjust $\psi_i$'s to be contractive.
\end{proof}

A representation $\sigma: A\to \mathbb B(H)$ is called essential if $\pi(A)\cap \mathbb K(H)=0$. A u.c.p. map $\phi: A\to \mathbb B(H)$ is called a representation modulo compacts if $\pi\circ\phi$ is a $*$-homomorphism, where $\pi:\mathbb B(H)\to\mathbb B(H)/\mathbb K(H)$ is the quotient map. For a map $V: K\to H$ between Hilbert spaces (or Hilbert $C^*$-modules we use the usual notation $ad_V: \mathbb B(H)\to \mathbb B(K)$ to denote the adjoint map: $x\mapsto V^*xV$. The next result is \cite[Theorem 3.2]{m}, adapted to the notations of the current paper.

\begin{theorem} [Voiculescu extension theorem] \label{voic}
Let $\mathfrak A$ be a finite atomic von Neumann algebra and $A$ be a separable $C^*$-algebra and unital central left $\mathfrak A$-module. Let $K$ be a separable Hilbert space, on which $\mathfrak A$ acts via a faithful representation, and $\varphi: A\to \mathbb B(K)$ be a faithful representation modulo compacts and a left $\mathfrak A$-module map, and $\sigma: A\to  \mathbb B(H)$ be a faithful unital essential representation. Let $\tau:\mathfrak A\to A; \alpha\mapsto \alpha\cdot 1_A$ be faithful and give $H$ the structure of a right Hilbert $\mathfrak A$-module via the faithful representation $\sigma\circ\tau: \mathfrak A\to \mathbb B(H)$. Assume that the commutant of $\sigma\circ\tau(\mathfrak A)$ in $\mathbb B(H)$ is not a finite von Neumann algebra. Let $B\subseteq \mathbb B(H)$ be a separable $C^*$-subalgebra and a left $\mathfrak A$-submodule. Then there is a net $V_i: K\to H$ of module map isometries such that $(ad_{V_k}- \varphi)(B)\subseteq \mathbb K(K)$, for each $k$, and $ad_{V_k}\to \varphi$ on $A$ in the point-norm topology.
\end{theorem}

\bibliographystyle{line}
\bibliography{JAMS-paper}

\end{document}